\documentclass[10pt]{amsart}
\usepackage{enumerate,amsmath,amssymb,latexsym,
amsfonts, amsthm, amscd}

\theoremstyle{plain}

\newtheorem{theorem}{Theorem}

\numberwithin{equation}{section}

\newcommand{\ra}{\rightarrow}

\newcommand{\R}{\mathbb{R}}

\newcommand{\mdot}{\,\begin{picture}(-1,-1)(-1,-1)\circle*{2}\end{picture}\ }

\addtocounter{section}{-1}

\begin{document}

\title { The triple-zero Painlev\'e I transcendent}

\date{}

\author[P.L. Robinson]{P.L. Robinson}

\address{Department of Mathematics \\ University of Florida \\ Gainesville FL 32611  USA }

\email[]{paulr@ufl.edu}

 \keywords{}

\begin{abstract}

We offer elementary proofs for fundamental properties of the unique triple-zero solution to the first Painlev\'e equation.

\end{abstract}

\maketitle

\medbreak

\section*{Introduction} 

\medbreak 

The six Painlev\'e equations emerged from the research of Painlev\'e, Gambier and Fuchs in answer to a question of Picard as to when the solutions of certain second-order ordinary differential equations have only poles among their movable singularities; solutions of these six equations are often referred to as Painlev\'e transcendents. These ODEs play important r\^oles in numerous diverse fields: for example, they appear as self-similar reductions of PDEs including the Korteweg de Vries equation, whose birth was well nigh coincident with the birth of the Painlev\'e equations themselves. 

\medbreak 

As befits their importance, the Painlev\'e equations have been subjected to intense scrutiny using highly sophisticated techniques. Our purpose here is extremely modest: we consider only the first Painlev\'e equation
\begin{equation} \label{P_I}
w''(z) = 6 w(z)^2 + z \tag{{\bf PI}}
\end{equation}
and only as a real equation; moreover, we only consider the unique solution of this equation that has a third-order zero at the origin. In line with these modest aims, we establish some of the most fundamental properties of this triple-zero Painlev\'e transcendent, giving explicit proofs and using elementary methods. 

\medbreak

\section*{The Transcendent} 

\medbreak 

Symmetries of the first Painlev\'e equation \ref{P_I} are important and will be useful in what follows. Let $w$ be any solution of \ref{P_I}; let $\alpha$ and $\beta$ be nonzero scalars. If $W(z) = \beta w (\alpha z)$ then 
$$\beta W''(z) = 6 \alpha^2  W(z)^2 + \alpha^3 \beta^2 z$$
by direct calculation. In particular, if $\alpha$ is a fifth-root of unity and $\beta = \alpha^2$ then $W$ is also a solution of \ref{P_I}. To put this another way, if $w$ is a solution of \ref{P_I} and $\gamma$ is a fifth-root of unity then the rule $W(z):= \gamma^2 w(\gamma z)$ also defines a solution of \ref{P_I}; thus the group of fifth-roots of unity acts on the set of Painlev\'e I transcendents. This being so, it is natural to ask whether there exist Painlev\'e I transcendents that are invariant under this action. An examination of Taylor series about the origin reveals two invariant transcendents: one of these has a double pole at the origin; the other has a triple zero at the origin (and will be our focus, as follows). 

\medbreak 

By the {\it triple-zero Painlev\'e I transcendent} we shall mean the unique solution of \ref{P_I} that vanishes along with its derivative at the origin. As indicated in our Introduction, we shall regard the triple-zero Painlev\'e I transcendent as a real function $s$ of a real variable $t$ and shall take it as being defined on its maximal domain, an open interval of reals containing the origin. Further, we shall write $\overset{\mdot}{s}$  rather than $s'$ for the derivative; this allows $\overset{\mdot}{s}\:^ 2$  rather than $s'^2$ or $(s')^2$ for its square. Note that $s(0) = \overset{\mdot}{s} (0) = \overset{\mdot \mdot}{s}(0) = 0$ and that $\overset{\mdot \mdot \mdot}{s}(0) = 12 s (0) \overset{\mdot}{s}(0) + 1 = 1 \neq 0$ so that $s$ indeed has a triple zero at the origin. In fact, the only point at which a solution of \ref{P_I} can have a third-order zero is the origin, as is plain from the equation itself; accordingly, the stipulation `triple-zero' singles out $s$ from among all Painlev\'e I transcendents.

\medbreak 

Actually, we shall find it convenient to cast \ref{P_I} into equivalent forms by rescaling. The choice $\alpha = 6^{1/5}, \beta = 6^{2/5}$ casts \ref{P_I} in the form 
$$W''(z) = 6 W(z)^2 + 6 z$$
used by Boutroux in his asymptotic analysis of the first Painlev\'e equation. In the section on `{\it Finite-time Blow-up}' we shall take $s$ to be the triple-zero solution of {\it this} equation and examine its behaviour for nonnegative real arguments. The choice $\alpha = - 6^{1/5}, \beta = - 6^{2/5}$ casts \ref{P_I} in the form
$$W''(z) = 6 z - 6 W(z)^2 $$
and in the section on `{\it Oscillatory Behaviour}' we take $s$ to be the triple-zero solution of {\it this} equation, analyzing it for nonnegative real arguments. Here, the rescaling has a definite aesthetic advantage: the discussion of this form leads to expressions in $\sqrt t$; discussion of the original form \ref{P_I} involves the unsightly $\sqrt { - t/6}$. Taken together, these two sections cover the triple-zero Painlev\'e I transcendent on the whole of its maximal real domain.  

\medbreak 

\section*{Finite-time Blow-up} 

\medbreak 

In this section, $s$ will denote the unique solution of the first Painlev\'e equation in the form 
\begin{equation} \label{PI+}
\overset{\mdot \mdot}{s}(t) = 6 s(t)^2 + 6 t \tag{{\bf PI+}}
\end{equation}
with initial data $s(0) = \overset{\mdot}{s} (0) = 0$ so that $\overset{\mdot \mdot}{s}(0) = 0$ and $\overset{\mdot \mdot \mdot}{s}(0) = 6 > 0$. We regard $s$ as being defined on its maximal real domain, which is an open interval $I$ about the origin. Our purpose here is to consider the behaviour of $s$ on $I \cap [0, \infty)$. 

\medbreak 

It is at once clear that both the transcendent $s$ and its derivative $\overset{\mdot}{s}$ are strictly positive at each strictly positive point in the maximal real domain $I$: in fact, if $t > 0$ lies in $I$ then $\overset{\mdot \mdot}{s}(t) = 6 s(t)^2 + 6 t > 6 t$ so that successive  integrations yield $\overset{\mdot}{s}(t) > 3 t^2$ and $s(t) > t^3$ when the initial data are taken into account. 

\medbreak 

We claim that the transcendent $s$ blows up in finite time: there exists $t_{\infty} > 0$ such that $s(t) \rightarrow \infty$ as $t \uparrow t_{\infty}$; thus, $I \cap [0, \infty) = [0, t_{\infty})$. 

\medbreak

\begin{theorem} \label{+upper}
If $\tau > 0$ is an arbitrary positive element of $I$ then $\tau + s(\tau)^{-1/2}$ lies outside $I$. 
\end{theorem} 

\begin{proof} 
 If $t > 0$ is in $I$ then $\overset{\mdot \mdot}{s}(t) = 6 s(t)^2 + 6 t > 6 s(t)^2$; thus $2 \overset{\mdot}{s} \: \overset{\mdot \mdot}{s}> 12 s^2 \overset{\mdot}{s}$ so that $\overset{\mdot}{s}^2 - 4 s^3$ has strictly positive derivative and is therefore strictly increasing on $I \cap [0, \infty)$. Consequently, 
$$0 < t \in I \Rightarrow \overset{\mdot}{s}(t) > 2 s(t)^{3/2}$$
or 
$$0 < t \in I \Rightarrow -  \frac{1}{2} s(t)^{-3/2} \overset{\mdot}{s}(t) < - 1.$$
Now, let $\tau > 0$ be an arbitrary positive element of $I$. If $\tau < t \in I$ then by integration
$$s(t)^{-1/2} - s(\tau)^{-1/2} < - t + \tau$$
so that 
$$s(t)^{1/2} > \frac{1}{s(\tau)^{-1/2} + \tau - t}.$$
Here, the right side blows up at $t = \tau + s(\tau)^{-1/2}$; consequently, the left side blows up at this time or sooner. 
\end{proof}

\medbreak

Thus: if $0 < \tau \in I$ then $\tau < t_{\infty} \leqslant \tau + s(\tau)^{-1/2}$ and therefore $\tau < t_{\infty} < \tau + \tau^{-3/2}$. 

\medbreak 

We can place a definite lower bound on the finite time to blow-up as follows. 

\medbreak 

\begin{theorem} \label{+lower}
The blow-up time $t_{\infty}$ satisfies 
$$t_{\infty} > \sqrt{\frac{3}{2}} \arctan \sqrt{\frac{2}{3}} + \frac{2}{3} \log \frac{5}{2}. $$
\end{theorem} 

\begin{proof} 
If $0 < t \in I$ then $s(t) > t^3$ so that \ref{PI+} yields
$$\overset{\mdot \mdot}{s} (t) < 6 s(t)^2 + 6 s(t)^{1/3}$$
whence multiplication by $2 \overset{\mdot}{s}(t)$ and integration from $0$ to $t$ yield 
$$\overset{\mdot} {s} (t)^2 < 4 s(t)^3 + 9 s(t)^{4/3}.$$
Complete the square, pointwise adding $2 \cdot 2 \cdot 3 \cdot s^{3/2} \cdot s^{2/3}$ to each side: 
$$\overset{\mdot}{s} \: ^2 + 12 s^{13/6} < (2 s^{3/2} + 3 s^{2/3})^2$$
so that certainly 
$$\overset{\mdot}{s} < 2 s^{3/2} + 3 s^{2/3}.$$
Let $t$ run from $0$ to $T$ and $s$ run from $0$ to $S$ correspondingly: integration then yields 
$$T > \int_0^S  \frac{{\rm d} s}{2 s^{3/2} + 3 s^{2/3}};$$
the limits $T \uparrow t_{\infty}$ and $S \uparrow \infty$ correspond, so 
$$t_{\infty} \geqslant \int_0^{\infty} \frac{{\rm d} s}{2 s^{3/2} + 3 s^{2/3}}.$$
Rather than evaluate this lower bound for $t_{\infty}$ directly, we settle for a reasonable lower estimate. 
For $0 < s < 1$ we substitute $s = r^3$ and note that $r^{5/2} < r^2$ to derive 
$$\int_0^1 \frac{{\rm d} s}{2 s^{3/2} + 3 s^{2/3}} = \int_0^1 \frac{{\rm d} r}{\frac{2}{3} r^{5/2} + 1} > \int_0^1 \frac{{\rm d} r}{\frac{2}{3} r^2 + 1} = \sqrt{\frac{3}{2}} \arctan \sqrt{\frac{2}{3}}.$$ 
For $s > 1$ we note that $s^{2/3} < s$ and substitute $s = r^2$ to derive
$$\int_1^{\infty} \frac{{\rm d} s}{2 s^{3/2} + 3 s^{2/3}} > \int_1^{\infty} \frac{{\rm d} s}{2 s^{3/2} + 3 s} = \int_1^{\infty} \frac{{\rm d} r}{r^2 + \frac{3}{2} r} = \frac{2}{3} \log \frac{5}{2}. $$
\end{proof} 

\medbreak

Theorem \ref{+upper} tells us that if $0 < \tau \in I$ then $\tau < t_{\infty} < \tau + \tau^{-3/2}$ while Theorem \ref{+lower} places $\sqrt{3/2} \arctan \sqrt{2/3} + (2/3) \log (5/2) = 1.449...\:$ in $I$.  The expression $\tau + \tau^{-3/2}$ has its minimum value $(3/2)^{2/5} + (2/3)^{3/5}$ at $(3/2)^{2/5} = 1.176 ... \in I.$ Accordingly, $(3/2)^{2/5} + (2/3)^{3/5} = 1.960 ...\:$ is a definite upper bound on $t_{\infty}$. 

\medbreak 

These inexpensive definite bounds $1.449... < t_{\infty} < 1.960 ...\:$ are not too bad; closer analysis shows that finite-time blow-up occurs between 1.82 and 1.83.

\medbreak 

\section*{Oscillatory Behaviour} 

\medbreak 

In this section, $s$ will denote the unique solution of the first Painlev\'e equation in the form 
\begin{equation} \label{PI-}
\overset{\mdot \mdot}{s}(t) = 6 t - 6 s(t)^2 \tag{{\bf PI-}}
\end{equation}
with initial data $s(0) = \overset{\mdot}{s} (0) = 0$. Again we regard this $s$ as defined on its maximal real domain, which is now the open interval $J = \{ - t : t \in I \}$. Our purpose here is to consider the behaviour of $s$ on $J \cap [0, \infty) = [0, \infty)$, to which the discussion in this section is restricted.

\medbreak 

The following first integral of \ref{PI-} will generate some useful estimates. 

\medbreak 

\begin{theorem} \label{int}
If $t > 0$ then 
$$\overset{\mdot}{s}(t)^2 + 4 s(t)^3 + 12 \int_0^t s = 12 t s(t).$$
\end{theorem} 

\begin{proof} 
Simply multiply \ref{PI-} throughout by $2 \overset{\mdot}{s}(t)$ and then integrate from $0$ to $t$ while invoking $s(0) = \overset{\mdot}{s}(0) = 0$. 
\end{proof} 

\medbreak 

The circumstance $\overset{\mdot \mdot \mdot}{s}(0) = 6 > 0$ implies that the triple-zero transcendent $s$ is initially increasing and therefore initially positive. An immediate first corollary of Theorem \ref{int} is that this positivity persists. 

\medbreak 

\begin{theorem} \label{pos}
If $t > 0$ then $s(t) > 0$. 
\end{theorem} 

\begin{proof} 
Aim at a contradiction by supposing that $s$ has a positive zero, in which case $\tau := \inf \{ t > 0 : s(t) = 0 \}$ satisfies $s(\tau) = 0$ and is itself strictly positive. As $s$ is strictly positive on the interval $(0, \tau)$ it follows from Theorem \ref{int} that 
$$0 = 12 \tau s(\tau) - 4 s(\tau)^3 = \overset{\mdot}{s}(\tau)^2 + 12 \int_0^{\tau} s > 0.$$
This contradiction faults the supposition and concludes the proof. 
\end{proof} 

\medbreak 

A further corollary is an upper bound on the transcendent. 

\begin{theorem} \label{upper}
If $t > 0$ then $s(t) < \sqrt{3 t}$. 
\end{theorem} 

\begin{proof} 
Bearing Theorem \ref{pos} in mind, if $t > 0$ then Theorem \ref{int} yields $4 s(t)^3 < 12 t s(t)$ whence $s(t)^2 < 3 t$ and $s(t) < \sqrt{3 t}$ as asserted. 
\end{proof} 

\medbreak 

We may also establish additional bounds, as follows. 

\medbreak 

\begin{theorem} \label{3/28}
If $t > 0$ then $t^3 > s(t) > t^3 - \frac{3}{28} t^8$. 
\end{theorem} 

\begin{proof} 
If $t > 0$ then  $\overset{\mdot \mdot}{s}(t) = 6 t - 6 s(t)^2 < 6 t$ whence $\overset{\mdot}{s}(t) <  3 t^2$ and $s(t) < t^3$ by successive integrations; in turn, $\overset{\mdot \mdot}{s}(t) = 6 t - 6 s(t)^2 > 6 t - 6 t^6$
whence successive integrations yield $\overset{\mdot}{s}(t) > 3 t^2 - \frac{6}{7} t^7$ and $s(t) > t^3 - \frac{3}{28} t^8$. 
\end{proof} 

\medbreak 

Now, the equation \ref{PI-} compels each of its solutions to inflect on the parabola `$s^2 = t$', to be concave up in the region `$s^2 < t$' and to be concave down in the region `$s^2 > t$'. In particular, this is true of the triple-zero transcendent $s$: it is initially concave up and increasing; this increase continues until the vanishing of $\overset{\mdot}{s}$; the mean value theorem ensures that this cannot occur until after the vanishing of $\overset{\mdot \mdot}{s}$; and this last event does occur, because the square-root function is concave down. 

\medbreak 

\begin{theorem} \label{t0}
There exists a least $t_0 > 0$ such that $s(t_0) = \sqrt{t_0}$. This $t_0$ satisfies 
$$1 < t_0 < (5/4)^{2/5} = 1.093...\:.$$
\end{theorem} 

\begin{proof} 
Consider the first coincidence of $s(t)$ and $\sqrt{t}$. As $s(t) < t^3$, this first coincidence occurs after $t = 1$. As $s(t) > t^3 - \frac{3}{28} t^8$, it occurs before $t > 0$ satisfies $t^3 - \frac{3}{28} t^8 = \sqrt{t}$: seek the least value of $r := \sqrt{t} > 0$ such that $p(r) := 3 r^{15} - 28 r^5 + 28 = 0$; elementary arithmetic shows that $p((5/4)^{1/5}) = - \frac{73}{64} < 0$ while of course $p(1) = 3 > 0$. 
\end{proof} 

\medbreak 

From this point on, the behaviour of $s$ is more interesting. In brief, we expect the following: the transcendent $s$ has just entered the region `$s^2 > t$' in which it must be concave down; $s$ then returns to cross the parabola `$s^2 = t$' again, thereby entering the region `$s^2 < t$' in which it must be concave up. This oscillatory behaviour continues, $s$ crossing and recrossing the square-root  ad infinitum. We develop explicit proofs for some simple estimates regarding this behaviour by applying `Sturm comparison' techniques to the deviation of $s$ from the square-root function. 

\medbreak 

Let $f := s - \sqrt{}$ be the deviation of $s$ from the square-root. If $t > 0$ then 
$$\overset{\mdot \mdot}{f} (t) = \phi(t) - \Phi(t) f(t)$$
where 
$$\phi(t) = \frac{1}{4 t \sqrt{t}} > 0$$
and where 
$$\Phi(t) = 6(\sqrt{t} + s(t))$$
satisfies 
$$ 6 \sqrt{t} < \Phi(t) <  6(1 + \sqrt{3})\sqrt{t}$$
on account of Theorem \ref{pos} and Theorem \ref{upper}.

\medbreak 

We shall compare $f$ with a suitable trigonometric function $g$ satisfying $\overset{\mdot \mdot}{g} + \lambda g = 0$ for a specially chosen $\lambda > 0$. By direct calculation,  
$$(\overset{\mdot}{f} g - f \overset{\mdot}{g})^{\mdot} = \{ \phi + (\lambda - \Phi)f \} g$$
whence it follows that if $0 < a < b$ then 
\begin{equation} \label{FT}
\int_a^b \{ \phi + (\lambda - \Phi) f) \} g = [\overset{\mdot}{f} g - f \overset{\mdot}{g}]_a^b. \tag{{\bf FT}}
\end{equation}
In the next pair of proofs, we engineer contradictions by choosing $\lambda$ and arranging $g$ so that the left side is strictly positive while the right side is non-positive. 

\medbreak 

Our first comparison places a strict upper bound on the length of an interval over which $s$ falls short of the square-root function. 

\medbreak 

\begin{theorem} \label{0}
Let $0 < a < b$: if $s < \sqrt{}$ on the interval $(a, b)$ then 
$$b - a < \pi \: 6^{-1/2} \: a^{-1/4}.$$
\end{theorem} 

\begin{proof} 
Deny the conclusion: say $b - a \geqslant \pi \: 6^{-1/2} \: a^{-1/4}$ and aim at a contradiction. Let $\lambda > 0$ be given by $\pi / \sqrt{\lambda} = b - a$; it follows that $\lambda \leqslant 6 \sqrt{a}$.  Note that if $a \leqslant t$ then $\Phi(t) > 6 \sqrt{t} \geqslant 6 \sqrt{a}$ so that $\lambda < \Phi$ on $[a, b]$. The gap between consecutive zeros of a nontrivial solution $g$ to $\overset{\mdot \mdot}{g} + \lambda g = 0$ being $b - a$, we may choose such $g$ to satisfy $g > 0$ on $(a, b)$ in which case $g(a) = 0 = g(b)$ and $\overset{\mdot}{g}(a) > 0 > \overset{\mdot}{g}(b)$. The integral on the left side of \ref{FT} has integrand strictly positive on $(a, b)$ and is therefore strictly positive, whereas the right side of \ref{FT} reduces to $f(a) \overset{\mdot}{g}(a) - f(b) \overset{\mdot}{g}(b) \leqslant 0$ because $f \leqslant 0$ on $[a, b]$; this contradiction concludes the proof. 
\end{proof} 

\medbreak 

Our second comparison places a strict lower bound on the length of an interval over which $s$ exceeds the square-root. 

\medbreak 

\begin{theorem} \label{rt3}
Let $0 < a < b$:  if $s > \sqrt{}$ on the interval $(a, b)$  and if $s = \sqrt{}$ at its ends then
$$b - a > \pi \: (6(1 + \sqrt{3}))^{- 1/2} \: b^{-1/4}.$$
\end{theorem} 

\begin{proof} 
Deny the conclusion: say $b - a \leqslant  \pi \: (6(1 + \sqrt{3}))^{- 1/2} \: b^{-1/4}$ and aim at a contradiction. If $0 < t \leqslant b$ then $\Phi(t) < 6(1 + \sqrt{3})\sqrt{t} \leqslant 6(1 + \sqrt{3})\sqrt{b}$; the choice $\lambda = 6(1 + \sqrt{3})\sqrt{b}$ therefore ensures that $\lambda > \Phi$ on $[a, b]$. The gap $\pi/\sqrt{\lambda}$ between consecutive zeros of a nontrivial $g$ satisfying $\overset{\mdot \mdot}{g} + \lambda g = 0$ is now $b - a$ at least; consequently, we may adjust $g$ so as to arrange that $g > 0$ on $(a, b)$. The integral on the left side of \ref{FT} is again strictly positive, as its integrand is a strictly positive function on $(a, b)$. The right side of \ref{FT} reduces to $\overset{\mdot}{f}(b) g(b) - \overset{\mdot}{f}(a)g(a) \leqslant 0$ because $\overset{\mdot}{f}(a) \geqslant 0 \geqslant \overset{\mdot}{f}(b)$. This contradiction concludes the proof. 
\end{proof} 

\medbreak 

As the difference $s - \sqrt{}$ has constant sign between consecutive zeros, we may combine the results of these comparisons as follows: if $a < b$ are consecutive crossings of $s$ and the square-root function then 
$$b - a \geqslant \pi \: 6^{-1/2} \: a^{-1/4} \; \Rightarrow \; s > \sqrt{} \; \; {\rm on} \; \; (a, b) \; \Rightarrow \; b - a >  \pi \: (6(1 + \sqrt{3}))^{- 1/2} \: b^{-1/4}$$
while 
$$b - a \leqslant  \pi \: (6(1 + \sqrt{3}))^{- 1/2} \: b^{-1/4} \; \Rightarrow \; s < \sqrt{} \; \; {\rm on} \; \; (a, b) \; \Rightarrow \; b - a < \pi \: 6^{-1/2} \: a^{-1/4}.$$

\bigbreak

\section*{Remarks} 

\medbreak 

We draw our account to a close with some very brief remarks on improvements and on the further theory. 

\medbreak 

The estimates in Theorem \ref{3/28} can be sharpened. For instance, if $0 < t < (28/3)^{1/5}$ (so that $ t^3 - \frac{3}{28} t^8 > 0$) then we may insert $s(t) >  t^3 - \frac{3}{28} t^8$ in $\overset{\mdot \mdot}{s}(t) = 6 t - 6 s(t)^2$ to obtain 
the inequality $\overset{\mdot \mdot}{s}(t) < 6 t - 6 t^6 + \frac{9}{7} t^{11} - \frac{27}{392} t^{16}$ whence two integrations result in the squeeze
$$t^3 - \frac{3}{28} t^8 < s(t) <  t^3 - \frac{3}{28} t^8 + \frac{3}{364} t^{13} - \frac{3}{13328} t^{18}.$$
The leftmost polynomial here is the eighth-order Taylor polynomial of $s$, while the rightmost polynomial differs from the eighteenth-order Taylor polynomial of $s$ only in the coefficient of its last term. Naturally, this squeeze leads to a considerable tightening of the bounds on $t_0$ derived in Theorem \ref{t0}. 

\medbreak 

In Theorem \ref{0} and Theorem \ref{rt3} we addressed the oscillation of $s$ about the square-root function and we placed bounds on the times between successive oscillations, but we did not bound the amplitudes of these oscillations. Of course, we do have some information: Theorem \ref{pos} and Theorem \ref{upper} together ensure that if $t > 0$ then the ratio $s(t)/\sqrt{t}$ varies strictly between $0$ and $\sqrt{3}$; further inspection replaces this upper bound by $\sqrt{2}$. In fact, as $t$ tends to infinity, the ratio $s(t)/\sqrt{t}$ tends to unity and indeed the difference $s(t) - \sqrt{t}$ tends to zero. More precise information is available: the application of standard techniques from asymptotic analysis indicates that $s(t) - \sqrt{t}$ is dominated by a multiple of $t^{-1/8}$ as $t \ra \infty$; for a discussion of this, see [1]. 

\medbreak 

Our account of the transcendent $s$ suggests that the oscillatory behaviour has greater interest than the finite-time blow-up; this is because we limited our account to the reals. By viewing \ref{P_I} as a {\it complex} equation, we see beyond the finite-time blow-up and realize that $t_{\infty}$ is merely the first in an infinite sequence of second-order poles along the ray from the origin through $t_{\infty}$. Invariance of the complex $s$ under the fifth-roots of unity ensures that this behaviour is repeated along the rays from the origin making angles $\pm 2 \pi/5$ and $\pm 4 \pi/5$ with this one. More than this, the complex $s$ is meromorphic and has an infinite array of second-order poles in the open circular sectors of angle $2 \pi/5$ centred on these five rays. For an introduction to the complex theory, see [3]; [2] contains much more detail and many additional references on all six Painlev\'e equations.  

\bigbreak

\begin{center} 
{\small R}{\footnotesize EFERENCES}
\end{center} 
\medbreak

[1] C.M. Bender and S.A. Orszag, {\it Advanced Mathematical Methods for Scientists and Engineers I}, Mc Graw-Hill (1978); Springer (1999). 

\medbreak 

[2] V.I. Gromak, I. Laine and S. Shimomura, {\it Painlev\'e Differential Equations in the Complex Plane}, de Gruyter (2002). 

\medbreak 

[3] E. Hille, {\it Ordinary Differential Equations in the Complex Domain}, Wiley-Interscience (1976); Dover Publications (1997). 

\medbreak

\end{document}